\RequirePackage{fix-cm}
\documentclass[smallextended]{svjour3}       % onecolumn (second format)
\smartqed  % flush right qed marks, e.g. at end of proof
\usepackage{graphicx}
\usepackage{amsmath}
\usepackage{amssymb}
\newtheorem{thm}{Theorem}

\usepackage{color}

\newcommand{\ignore}[1]{}

\newcommand{\R}{\mathbb{R}}

\newcommand{\cP}{\mathcal{P}}

\newcommand{\x}{{\rm exc}}
\newcommand\ring[1]{\mathaccent23{#1}}
\def\Vr{\ring{V}} % {\mbox{\r{\itshape V}}
\RequirePackage[normalem]{ulem} %DIF PREAMBLE
\RequirePackage{color}\definecolor{RED}{rgb}{1,0,0}\definecolor{BLUE}{rgb}{0,0,1} %DIF PREAMBLE
\begin{document}

\title{Phylogenetic flexibility via Hall-type inequalities and submodularity}
%\subtitle{Do you have a subtitle?\\ If so, write it here}

%\titlerunning{Short form of title}        % if too long for running head

\author{Katharina T.~Huber \and Vincent Moulton \and Mike Steel}

%\authorrunning{Short form of author list} % if too long for running head

\institute{K. T.~Huber \and V. Moulton\at
           School of Computing Sciences, University of East Anglia, Norwich, UK \\
           \email{K.Huber@uea.ac.uk, V.Moulton@uea.ac.uk}           %  \\
%          \emph{Present address:} of F. Author  %  if needed
           \and
           M. Steel \at
           Biomathematics Research Centre, University of Canterbury, Christchurch, NZ\\
            \email{mike.steel@canterbury.ac.nz}  
}

\date{Received: date / Accepted: date}
% The correct dates will be entered by the editor

\maketitle

\begin{abstract}
Given a collection $\tau$ of subsets of a finite set $X$, we say that $\tau$ is {\em phylogenetically flexible} if, for any collection $R$ of rooted phylogenetic trees whose leaf sets comprise the collection $\tau$, $R$ is compatible (i.e. there is a rooted phylogenetic $X$--tree that displays each tree in $R$).  We show that $\tau$ is phylogenetically flexible if and only if it satisfies a Hall-type inequality condition of being `slim'. 
Using  submodularity arguments, we show that there is a polynomial-time algorithm for determining whether or not $\tau$ is slim.
This `slim' condition reduces to a simpler inequality in the case where all of the sets in $\tau$ have size 3, a property we call `thin'.  Thin sets were recently shown to be equivalent to the existence of an (unrooted) tree for which the median function provides an injective mapping to its vertex set;  we show here that the unrooted tree in this representation can always be chosen to be a caterpillar tree. We also characterise when a collection $\tau$ of subsets of size 2 is thin 
(in terms of the flexibility of total orders rather than phylogenies) and show that this holds if and only if an associated bipartite graph is a forest. The significance of our results for phylogenetics is in providing precise and efficiently verifiable conditions under which supertree methods that require consistent inputs of trees,  can be applied to  any input trees on given subsets of species.  
% \PACS{PACS code1 \and PACS code2 \and more}
% \subclass{MSC code1 \and MSC code2 \and more}
\end{abstract}

\section{Introduction}
\label{intro}
In phylogenomics,  biologists often encounter the following problem: Given a collection $\tau$ of  different subsets of species, the corresponding  phylogenetic trees --- each one reconstructed from the genomic data available for the corresponding subset ---  cannot be consistently combined  into a single phylogenetic tree for all the species.  
When this occurs, various heuristic and somewhat ad-hoc `supertree' methods (such as `matrix recoding with parsimony') are often applied to provide some estimate of a parent tree \cite{fel}.
However, when the collection of subsets of species has sufficiently sparse overlap (in a sense we will make precise shortly), then {\em any} phylogenetic tree assignment for $\tau$ will lead to a set of trees that can be consistently combined into a parent tree. Fig.~\ref{fig1}(i) provides an example of this.

\begin{figure}[htb]
\centering
\includegraphics[scale=1.0]{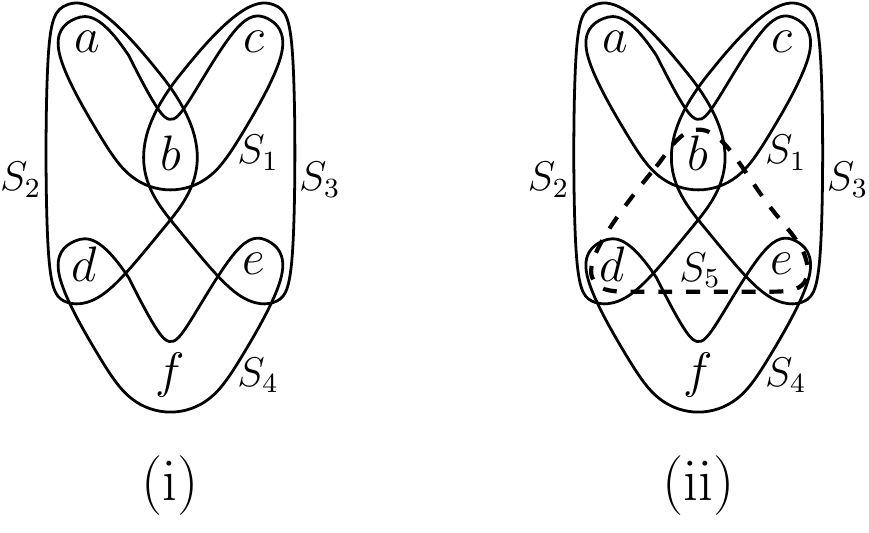}
\caption{A collection $\tau= \{\{a,b,c\}, \{a,b,d\}, \{b,c,e\}, \{d, e,
f\}\}$ of four sets that is phylogenetically flexible; and (ii)  a
collection $\tau' = \tau \cup \{\{b,d,e\}\}$ that fails to be
phylogenetically flexible.  In (i) all of the $3^4 = 81$ choices of rooted
triples (one for each of the four leaf sets) gives a set of rooted triples
that is displayed by at least one rooted phylogenetic tree on the six leaves
$a,b, \ldots, f$. However, in (ii) this fails, for example, the set of
rooted triples $ab|c, bd|a, bc|e, df|e$ together with $be|d$ (for the fifth
set) is not displayed by any tree on the six leaves. The set $\tau$ is thin,
but $\tau'$ is not, since it has a subset (namely $\tau'$ itself) which has
strictly negative excess (equal to $-1$).
}
\label{fig1}
\end{figure}

In this paper, we investigate the conditions under which the existence of a consistent parent tree can be guaranteed regardless of the tree structure for each subset. Here `parent' tree means that the leaf set of the tree is the union of the leaf sets of the input trees.   For example, given a set of input trees, 
if there is a parent tree that displays each tree, then a simple, fast and well-known algorithm due to \cite{aho} constructs such a tree in a canonical way.
However, this method will fail to return any phylogenetic tree when presented with input trees that are incompatible (i.e. cannot be displayed by any parent tree). In this paper, we characterise when such a method will always be safe to use on any set of input trees, given the sets of taxa that form the leaf sets of those trees. 
Thus, we consider as input just subsets of species, and develop mathematical characterisations and algorithms for this combinatorial question in the special case where each subset has a fixed (small) size. Later in the paper, we consider how the results extend to more general set systems. Our approach throughout is to  reduce  certain combinatorial questions in phylogenetics to the study of systems of inequalities involving linear expressions, and related submodularity properties.

In the discussion section, we mention a further biological context where the results may be relevant.  Note that there are many reasons why phylogenetic trees are constructed on different subsets of species, and a particularly topical one is that genes used to estimate a given phylogeny may only be present in (or have been sequenced) in a given subset of the species, and these subsets vary from gene to gene \cite{san11}.

Our work is motivated in part by a remarkable combinatorial result by Stefan Gr{\"unewald} \cite{gru12} involving unrooted binary trees. In that paper,
  a set $\cP$ of binary trees having leaves labelled from some set $X$ is said to be  `slim' if for every non-empty subset $\cP'$ of $\cP$, the number of leaves appearing in at least one tree in $\cP'$ is at least the total number of interior edges of $T$ plus 3.  Theorem~1.1 of \cite{gru12} then states that for any such  thin collection $\cP$ there is a tree with leaf set $X$ that `displays' each of the trees in $\cP$.  In particular, this leads to the rather striking consequence that ``the property of being slim only depends on the involved leaf sets of the trees and not on which phylogenetic tree is chosen for a fixed leaf set''(p. 324 \cite{gru12}). 
In this paper, we explore this notion further, and by working with rooted trees (rather than unrooted ones) we are able to establish precise characterizations of the analogous `slim' property.

Our work is also partly motivated by  results from \cite{dre09} where slim-type properties also arise in a tree-based setting, but for a quite different question involving `median' vertices. To explain this, given a tree $T = (V, E)$ and a subset $S$ of $V$ of size $3$, say $S = \{x, y, z\}$, consider the path in $T$ connecting $x,y$, the path
connecting $x, z$ and the path connecting $y, z$. There is a unique vertex that is shared by these three paths, the {\em median vertex}
of $S$ in $T$, denoted ${\rm med}_T (S)$.  In \cite{dre09}, the
authors show that  `slim'--type properties characterize when a set of
triples from $X$ can be realized as providing an encoding of the
interior vertices of a (unrooted) tree with leaf set $X$ 
 (an extension of this to sets of subset of $X$ of size greater than 3 is also described). In this paper, we extend this result further by showing that the tree that provides this encoding can be chosen to have a particular special type of structure (a `caterpillar'). 

The phylogenetic combinatorics of subsets of a species set is a topic that has  also been explored recently in the setting of `phylogenetic decisiveness' \cite{ste10}. However, the questions that we consider here are quite different from that setting;  rather than requiring a dense overlap of the species subsets in the phylogenetic decisiveness setting, here we investigate sparse overlap.  

We begin with some definitions.  Throughout this paper, $X$ will denote a fixed finite set.

\subsection{Thin set systems}

Suppose $\tau$ is a non-empty subset of $\binom{X}{r}$, $r \ge 2$. 
Let $L(\tau) =  \bigcup_{s\in \tau}s$ 
(i.e.  the set of elements of $X$ that appear in at least one set in $\tau$) and 
define the {\em excess} of $\tau$, denoted $\x(\tau)$,  by:
$$\x(\tau) = |L(\tau)|-|\tau| - (r- 1).$$

We  say that $\tau$ is {\em thin} if, for all non-empty subsets $\tau'$ of $\tau$, we have:
$$\x(\tau') \geq 0.$$

This notion appears in related but slightly different settings, namely for the leaf sets of unrooted trees in \cite{gru12},  in the median representation of sets of triples in \cite{dre09}, and as sparse triplet covers in \cite{gru17}.

In the following lemma, recall that a collection of (not necessarily distinct) sets $\{B_1, B_2, \ldots, B_m\}$ has a {\em system of distinct representatives} if one can select an element $x_i \in B_i$ for each $i \in \{1, \ldots, m\}$ so that the elements
$x_1, x_2, \ldots, x_m$ are all distinct. For  $\tau$ a non-empty subset of
 $\binom{X}{r}$, $r \ge 2$ with $L(\tau) = X$ and for $x\in X$ let 
$n_{\tau}(x)$ be the number of elements in $\tau$ that contain $x$.

\begin{lemma}
\label{lemthin}
	Let $\tau$ be a non-empty subset of $\binom{X}{r}$, $r \ge 2$ and $L(\tau) = X$.  If $\tau$ is thin, then the following properties hold:
	\begin{itemize}
	\item[(i)]  $|\tau| \leq n-r+1$ where $n=|X|$.
	\item[(ii)] For some $x \in X$, $n_{\tau}(x)\leq r-1$.
%$x$ is contained in at least $r-1$ elements of $\tau$. 
	\item[(iii)] For any subset $B$ of $X$ of size $r-1$, the collection of sets $\{S-B: S \in \tau\}$  has a system of distinct representatives.
	\end{itemize}
	\end{lemma}

	\begin{proof}
	Part (i)  follows from the defining condition for thin upon taking $\tau' = \tau$.  

Part (ii) can be established by the following double-counting argument. 	
%For $x \in X$, let $n_{\tau}(x)$ be the number of elements in $\tau$ that contain $x$.
Suppose that there is no element $x\in X$  with $n_{\tau}(x)\le r-2$, so 
that $n_{\tau}(x) \ge r-1$ for all $x \in X$. Let $\Omega = \{(x,S) \,:\, x \in S \in \tau\}$.
We then have:
\begin{equation}
\label{om1}
|\Omega| = \sum_{x \in X} n_{\tau}(x) \ge (r-1)k + r(n-k)
\end{equation}
where $k = |\{x \in X \,:\, n_{\tau}(x) = r-1\}|$.
On the other hand: 
\begin{equation}
\label{om2}
|\Omega| = r|\tau| \le r (n-(r-1)),
\end{equation}
 where the inequality is from Part (i).
Combining (\ref{om1}) and (\ref{om2}) gives $k \ge r(r-1)$ and, so,  $k\geq 2$. By the definition of $k$, (ii) follows.
%is no element $x \in X$ with $n_{\tau}(x) = r-1$.

For Part (iii), consider the union of any $l$ sets $A_1, A_2, \ldots,
  A_l $ where $ A_i =S_i - B$ and $S_i \in \tau$ for $i=1,\ldots,l$
  (note that these sets may have different sizes and a set may
  occur more than once). 
Since $\tau$ is thin, $|\bigcup_{i=1}^l S_i| \geq l + (r-1)$, and so, since $B$ has size $r-1$,
 $|\bigcup_{i=1}^lA_i| = |\bigcup_{i=1}^l S_i| -(r-1) \geq l$.
Since the inequality $|\bigcup_{i=1}^lA_i| \geq l$ holds for all $1\leq l
  \leq |\tau|$, Hall's marriage theorem \cite{hal} ensures that
$\tau$ has a system of distinct representatives. 

	\hfill$\Box$ \end{proof}

For the first part this paper, we will deal with the case where $r=3$. However,
the main theorem in this setting (Theorem~\ref{main})  will be used in Section~\ref{slimcase} to derive a result for the more general case where the sets have different sizes.
When $r=3$,  notice that if $|\tau'| =1$, then $\x(\tau') = 3-1-2=0$; however,
 if $|\tau'|=2$, then $\x(\tau') \geq 4-2-2=0,$
so it suffices, in the definition of thin, to consider subsets of $\tau'$ of $\tau$ of size at least 3.

A simple way to generate a thin set is to take any ordered sequence of subsets of $X$ of size 3, for which the ordered sequence has the property that each member contains
at least one element of $X$ that is not present in any earlier member of the sequence.  However, not all thin sets can be obtained in this way. For example, consider the collection 
$\{\{a,b,c\},   \{c, d, e\}, \{b, e, f\}, \{a, d, f\}\}$
of four subsets sets of $X=\{a, b, \ldots, f\}$.
This collection of subsets is thin, yet these four sets cannot be ordered so as to satisfy the property described.

\subsection{Phylogenetic trees and  flexible sets}

Following \cite{sem}, a {\em rooted phylogenetic tree}  $T$ is a 
rooted tree having a set $L(T)$ of labelled leaves (vertices of out-degree 0) and for which every non-leaf vertex is unlabelled and has out-degree at least 2.
We let $\rho_T$, or more briefly $\rho$ denote the root vertex of $T$, which has in-degree 0.  In case each non-leaf vertex has out-degree exactly 2 we say that $T$ is {\em binary}.  If $L(T)=X$, we will also say that $T$ is a {\em rooted phylogenetic $X$--tree}.  We let $\Vr(T)$ denote the set of interior (i.e. non-leaf) vertices of $T$. Similarly, an {\em unrooted phylogenetic tree}  $T$ is an 
unrooted tree having a set $L(T)$ of labelled leaves (vertices of degree 1) and for which every non-leaf vertex is unlabelled and has degree at least 3.  In case each non-leaf vertex has degree exactly 3 we say that $T$ is {\em binary}.  If $L(T)=X$, we will also say that $T$ is a {\em unrooted phylogenetic $X$--tree}.

A {\em rooted triple} is a rooted binary phylogenetic tree on three leaves, and we denote such a tree as $ab|c$ if it has leaf set $\{a,b,c\}$ with leaf $c$ adjacent to the root. A rooted phylogenetic $X$--tree $T$  is said to {\em display} the rooted triple $ab|c$ if some subdivision of the tree $ab|c$ is a subgraph of $T$.   

A {\em cherry} in a (rooted or unrooted) phylogenetic tree is a pair of
 leaves that is adjacent to the same vertex. A {\em rooted (respectively, 
unrooted) caterpillar} tree on $X$ is a rooted
(resp. unrooted) binary phylogenetic $X$--tree for which the number of cherries is
at most 1 (respectively, 2).

These notions are illustrated in Fig.~\ref{triplet}.

\begin{figure}[htb]
\centering
\includegraphics[scale=0.7]{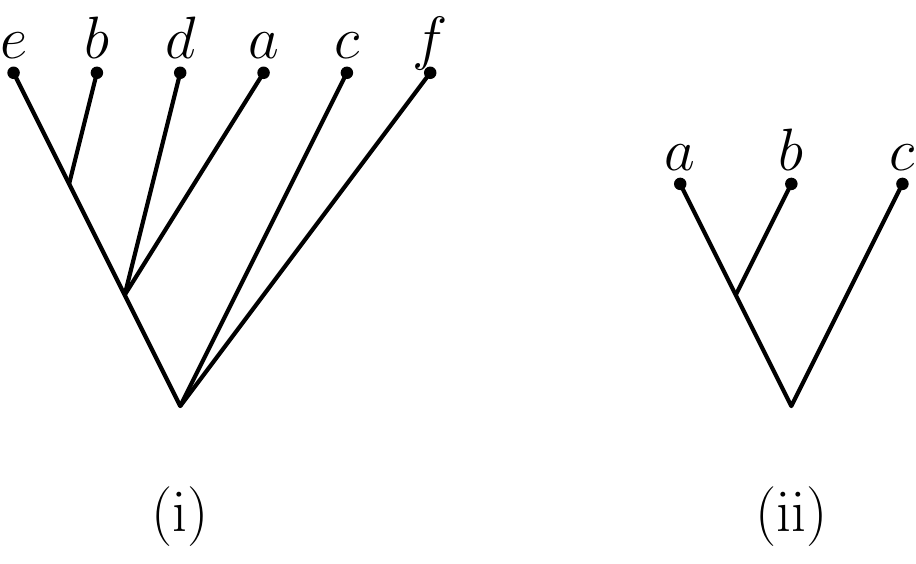}
\caption{(i) A rooted phylogenetic tree on leaf set $\{a,b,c, \ldots, f\}$. 
This tree is not binary, as it has a vertex of out-degree 3 (adjacent to $a$ 
and $d$). (ii) The rooted triple $ab|c$ for which $a,b$ forms a cherry.  This rooted
triple is also a rooted caterpillar and it is displayed by the tree in (i).}
\label{triplet}
\end{figure}

A set $R$ of rooted triples chosen from $X$ is said to be {\em compatible} if there is a rooted phylogenetic $X$--tree $T$ that displays each rooted triple in $R$ (in which case, we say that $T$ {\em displays} $R$).  Note that
if $R$ is compatible, then $T$ can always be chosen to be a binary tree and $R$ can contain at most one tree for any triplet (i.e. at most one of $ab|c$, $ac|b$, and $bc|a$ can be present in $R$). 

Suppose that we have a set $R$ of rooted triples with leaves chosen from $X$.  We will let $||R||$ denote the subset of $\binom{X}{3}$ consisting of the leaf sets of the trees in $R$.
We say that a non-empty subset $\tau$ of $\binom{X}{3}$ is {\em phylogenetically flexible} if every set $R$ of rooted triples for which $||R|| = \tau$ holds is compatible. An example to illustrate this notion is provided in Fig.~\ref{fig1}.

The following observation that phylogenetic flexibility is hereditary is straightforward to check.

\begin{lemma}\label{obs1}
	Suppose $\tau$ is a non-empty subset of ${X \choose 3}$ 
that is phylogenetically flexible.
	If $\tau'$ is a non-empty subset of $\tau$, then $\tau'$ is phylogenetically flexible. 
\end{lemma}

\section{Characterisation result}
We can now state our first main result.

\begin{thm}
\label{main}
Suppose that $\tau$ is a non-empty subset of $\binom{X}{3}$. Then  $\tau$ is phylogenetically flexible if and only if $\tau$ is thin.  
\end{thm}

The `if' direction of Theorem~\ref{main} can be established by applying Theorem 1.1 of \cite{gru12}; however, we give a shorter and more direct proof of this direction here (as well as establishing the converse). 
We begin with some preliminary results, which are required for the argument.

Given a rooted phylogenetic tree $T=(V,E)$ with leaf set $X$ 
and every vertex in $\Vr(T)-\{\rho_T\}$
having degree three. We say that a rooted
triple $xy|z$ {\em supports} 
a vertex $v$ in $T$ if $xy|z$ is displayed by $T$ and $v ={\rm  lca}_T(x,y)$. 

For a set $R$ of rooted triples on $X$, put $L(R)=\bigcup_{t\in R} L(t)$.
Furthermore, for a non-empty subset $S$ of $X$, let $[R, S]$ be the 
graph with vertex set $S$ and with an edge $\{a,b\}$ if and only if there exists
a  rooted triple $ab|c \in R$ for at least one element $c \in S$. 
By \cite[Theorem~2]{bry}, $R$ is compatible if and only if the graph $[R, S]$ is disconnected
for all subsets $S$ of $X$ of size at least 2.

\begin{lemma}\label{aho}
	Suppose that $T$ is a rooted binary phylogenetic $X$-tree,
%with all vertices in $\Vr(T)-\{\rho_T\}$ of degree 3, 
	that $R$ is a set of rooted triples with  $L(R) = X$,  and that each rooted triple supports 
	a unique (interior non-root) vertex in $T$.
	Then the graph $[R, X]$ has precisely two connected components.
\end{lemma}

\begin{proof}
For $v\in \Vr(T)-\{\rho_T\}$, let $X_v$ be the leaf set of the 
rooted subtree of $T$ 
with root $v$. 
We claim that for every such $v$, the graph induced by $[R,X]$ on 
$X_v$ is connected. The lemma then follows immediately by considering 
the graphs induced by $[R,X]$ on $X_u$, $X_w$ for $u$ and $w$ the 
children of the root of $T$.

To prove the claim, for $u$ (a child of the root $\rho_T$ of $T$),  we consider the following set:
$$
X^u=\{X_v  \,:\, v \mbox{ is an internal vertex of $T$ below or equal to } u\},
$$
where $v$ is said to be {\em below} $u$ if $u$ lies on the
path from $\rho_T$ to $v$.
Note that since $|X|\geq 3$, there must exist a child $u$ of $\rho_T$ 
such that $X^u\not=\emptyset$ and also there exists some
vertex $v\in V(T)$ below or equal to $u$ such that $|X_v|\geq 2$.
We use induction on $|X_v|$ for $X_v$ in $X^u$. If $|X_v|=2$, then both children of $v$ are leaves and
the lemma holds  because if $X_v = \{p,q\}$, 
then, by assumption, there exists a rooted triple
in $R$ of the form $r|pq$ for some $r\in X-\{p,q\}$ that supports $v$. Hence, there is an edge $\{p,q\}$ in $[R,X]$ and therefore the graph induced by $[R,X]$ on $X_v$ is connected.

Now suppose that $v$ 
is an internal vertex of $T$ below or equal to $u$ 
such that $|X_v|\geq 3$. Then at least one of the two 
children $v_1$ and $v_2$ of $v$ is not a leaf of $T$. 
Without loss of generality, we may assume
that $v_1$ is that child. Therefore,  $2\leq |X_{v_1}|< |X_v|$
and so, by induction, the graph induced by 
$[R,X]$ on $X_{v_1}$ is connected.  If $v_2$ is not a 
leaf of $T$, then the same arguments as before imply that
the graph induced by 
$[R,X]$ on $X_{v_2}$ is also connected. If $v_2$ is a
leaf of $T$, then the graph $[R,X]$ on $v_2$ is a vertex
and therefore is (trivially) connected. Since, by assumption,
there exists a rooted triple in $R$ that supports $v$, 
there is an  edge $\{y,z\}$ in $[R,X]$ with $y \in X_{v_1}$ and
$z \in X_{v_2}$. Hence the graph induced by $[R,X]$ on $X_v$ is connected.
\hfill$\Box$ \end{proof}

{\em Proof of Theorem~\ref{main}}

We first establish the `if' direction. Suppose that $\tau$ is thin, and let
$R$ be a set of rooted triples with leaves chosen from $X$ with $||R||=\tau$.   We show that any such choice of  $R$ is compatible. 

We will establish the compatibility of $R$ via the aforementioned
characterisation that $R$ is compatible if and only if $[R, S]$ is disconnected
for all subsets $S$ of $X$ of size at least 2. 
To that end, let $S$ be a subset of $X$ of size at least two.

Notice that $[R, S] = [R_S, S]$ where $R_S$ is the subset of those rooted triples in $R$ 
that have all three of their leaves in $S$. 
Let $\tau'= ||R_S||$.
Since $\tau$ is thin, we have $\x(\tau') \geq 0$, in other words:
\begin{equation}
\label{leq}
|L(\tau')|-|\tau'| \geq 2.
\end{equation}
Now (i) the number of vertices of $[R,S]$ is $|S|$ 
and $|S| \geq |L(\tau')|$; and 
(ii) the number of edges of $[R,S]$ is at most $|R_S| = |\tau'|$. 
Thus, by Inequality (\ref{leq}),  the number of vertices of $[R,S]$ minus the number of edges of this graph is at least 2. 
But any finite graph with this property must be disconnected. 
Since this holds for all subsets $S$ of $X$ of size at least two it follows that $R$ is compatible. 

We turn now to the `only if' direction.

We use induction on $|\tau|$. 
If $|\tau| = 1$ then $\tau$ is clearly thin.   So, 
suppose the `only if' direction  holds for all $\tau'\subset {X\choose 3}$ with 
$1\leq |\tau'|<m$, some $m \ge 2$, and let 
$\tau\subseteq {X\choose 3}$ such that $|\tau|=m$. 
Without loss of generality we may assume that $X = L(\tau)$.

Suppose that $\tau'$ is a non-empty proper subset of $\tau$. 
By Lemma~\ref{obs1}, $\tau'$ is
phylogenetically flexible. Hence by induction, $\tau'$ is thin. 
Thus, $|L( \tau')| \ge |\tau'| + 2$. To show that $\tau$ is thin,
it therefore suffices to prove that $|L(\tau)| \ge |\tau| + 2$.

Suppose for the purposes of obtaining a contradiction that 
$|L( \tau)| < |\tau| + 2$. Let $\{x,y,z\} \in \tau$ and set $\tau' = \tau - \{\{x,y,z\}\}$. Then, as $\tau'$ is thin by induction, 
\begin{equation}
%\label{longeq}
|\tau| +2>
|L(\tau)| = |L(\tau')| + (3 - |L(\tau') \cap \{x,y,z\}|) \ge |\tau| + 4 - |L(\tau') \cap \{x,y,z\}|.
\end{equation}
Hence $|L(\tau')\cap \{x,y,z\}|>2$ and, so,
$\{x,y,z\}\subseteq L(\tau')$. Thus, $L(\tau')=X$.

Now, since 
$\tau'$ is thin, there exists a (unrooted) phylogenetic 
tree $T=(V,E)$ with leaf set $X$,
and all vertices in $\Vr(T)$ of degree 3, for 
which the map ${\rm med}_T:\tau'\to \Vr(T)$ 
is one-to-one \cite{dre09} (see also 
%should probably insert reference for section 3 in main file and refer to this
Section~3  below). 
We claim that the map ${\rm med}_T$ must in fact be
bijective. Suppose that this is not the case. Then there exists some $v\in \Vr(T)$
such that ${\rm med}_T(s)\not=v$, for all $s\in \tau'$.
Hence, $|X|-2=|\Vr(T)|>|\tau'|$ and, so, $|X|-1>|\tau|$.
But then $|X|+1>|\tau|+2>|L(\tau)|=|X|$, which is impossible as $|\tau|+2$ is an integer.
Hence ${\rm med}_T$ is a bijection as claimed.

Now, root the tree $T$ by inserting a root vertex $\rho$ into an edge 
which separates $x,y$ from $z$, when the edge is removed from $T$.
Let $R'$ be a set of rooted triples induced by
the map ${\rm med}_T$ (for each element $\{a,b,c\}$ in $\tau'$,  ${\rm med}_T$ maps to some $v \in \Vr(T)$ so that we 
get a rooted triple with leaf set $\{a,b,c\}$ which supports $v$ in the 
rooted version of $T$) with $||R'||=\tau'$ and $L(R') = X$. Since ${\rm med}_T$ is a bijection, $R'$ 
satisfies the conditions of Lemma~\ref{aho} for the rooted version of $T$. 
Hence the graph $[R',X]$ has two connected components, 
one that contains $x,y$ in its vertex set and the other that contains $z$.

Now consider the set of rooted triples $R = R'  \cup \{y|zx\}$. 
Then $L(R)=X$, $[R,X]$ is connected and so $R$ is not 
compatible, and $||R|| = \tau$. But this is impossible, since $\tau$ is phylogenetically flexible.
\hfill$\Box$\\

The following corollary of Theorem~\ref{main} is now immediate from Lemma~\ref{lemthin}(i).

\begin{corollary}
\label{corby}
If  a  non-empty subset $\tau$ of $\binom{X}{3}$ is phylogenetically flexible, then $|\tau| \leq n-2$ where $n=|X|$.
\end{corollary}

We end this section by considering how many trees can display a set of rooted triples $R$ when $||R||$ is phylogenetically flexible.  It might be suspected that since the overlap between the leaf sets of the trees in $R$ is sparse, the number of trees displaying $R$ would need to be large. Indeed, this is sometimes the case; for example, suppose that the leaf sets in $R$ are all disjoint, so the total number of leaves is given by $n=3k$, where $k=|R|$. In this case, the number $N$ of rooted binary trees on $n$ leaves that display $R$ is given by:
\begin{equation}
\label{expeq}
N = \frac{(2n-3)!!}{3^{n/3}}, 
\end{equation}
which grows exponentially with $n$.
The proof of Eqn.~(\ref{expeq}) is to observe that each of the $3^k$ ways to select a rooted triple from the $k$ triples in $||R||$ provides a set of rooted triples that is displayed by at least one rooted phylogenetic tree (by the algorithm from \cite{aho}) and hence by at least one rooted binary tree, and these rooted binary trees are pairwise distinct, since any two of them display a different rooted triple for at least one triple in $||R||$.

At the other extreme, if $R$ has the maximum possible size for a phylogenetically flexible set on $n$ leaves (namely $n-2$ by Corollary~\ref{corby}), then it is possible for there to be  just a single rooted phylogenetic tree that displays 
$R$; this  is stated more precisely in the next proposition.   
\begin{proposition}
\label{propmike}
\mbox{}
\begin{itemize}
\item[(i)]  For every rooted binary phylogenetic $X$--tree $T$ on $n\geq 3$ leaves, there exists a set $R_T$ of $n-2$ rooted triples for which (a) $T$ is the only phylogenetic $X$-tree that displays $R_T$ and (b) 
$||R_T||$ is thin. 
\item[(ii)] There exist phylogenetically flexible sets of triples of size $n-2$ on $n$ leaves ($n \geq 6$) for which each assignment of a tree structure to these triples leads to a set of rooted triples that can be displayed by more than one rooted phylogenetic tree.
\end{itemize}
\end{proposition}

\begin{proof}
(i) We use induction on $n$.  For $n=3$, we can write $T=ab|c$, in which case $R_T=\{ab|c\}$ satisfies Conditions (a) and (b). Suppose now that Proposition~\ref{propmike} holds for $k\leq n$ where $n \geq 3$, and that $T$ is a rooted binary phylogenetic $X$--tree with $n+1$ leaves.  Select a  pair of leaves $a, b$ that are adjacent to the same vertex (say $v$) of $T$ (i.e. $\{a,b\}$ is a cherry of  $T$),  let vertex $u$ be the parent of  vertex $v$ in $T$, and  let $c$ be any leaf of $T$ present in the component of $T-u$ (the graph obtained by deleting $u$ from $T$)  that contains neither the root, nor the leaves $a, b$. 
Put
$X'=X-\{a\}$
and let $T'$ be the rooted binary phylogenetic $X'$-tree obtained from $T$ by deleting leaf $a$ and its incident edge, and suppressing the resulting vertex of degree 2. Since $T'$ has $n$ leaves, the induction hypothesis ensures  that there is a set $R_{T'}$ of $n-2$ rooted triples for which $T'$ is the only phylogenetic $X'$-tree that displays $R_{T'}$ and  that $||R_{T'}||$ is thin.  If we now let $R_T = R_{T'} \cup\{ab|c\}$, then $R_T$ is  a set of $(n+1)-2$ rooted triples and $R_T$ satisfies Conditions (a) and (b) for the tree $T$. This establishes the induction step and thereby the proposition.

(ii) Let $\tau = \{\{1,2,j\}: 2<j\leq n\}$. In this case,  $\tau$ is a thin (and therefore phylogenetically flexible) set of size $n-2$.
Now,  for $n\geq 6$, it can be checked that any assignment of a tree structure to these triples leads to a set of rooted triples that can be displayed by more than one rooted phylogenetic tree.
\hfill$\Box$ \end{proof}

\section{Median characterisations}

Given a phylogenetic tree $T$ with leaf set $X$ and a set $s \in \binom{X}{3}$, let
${\rm med}_T(s)$ refer to the vertex that is the unique median vertex of $T$ for the three elements of $s$.

The following result was established in
 \cite[Theorem 1.1]{dre09}. Suppose that 
 $\tau$  is a subset of $\binom{X}{3}$ with 
$L(\tau) = X$.  The following are equivalent:
\begin{itemize}
\item[(i)] $\tau$ is thin.
\item[(ii)] There exists a binary unrooted phylogenetic $X$--tree $T=(V,E)$
for which the function 
${\rm med}_T:\tau \to \Vr(T)$: 
$s \mapsto {\rm med}_T(s)$ from the elements $s$ of $\tau$ to the set of interior  vertices of $T$ is one-to-one.
\end{itemize}
When (ii) holds, we say that $T$ provides a {\em median representation} of $\tau$.
Fig. ~\ref{fig3}(i) illustrates how this equivalence applies. 

\begin{figure}[htb]
\centering
\includegraphics[scale=0.7]{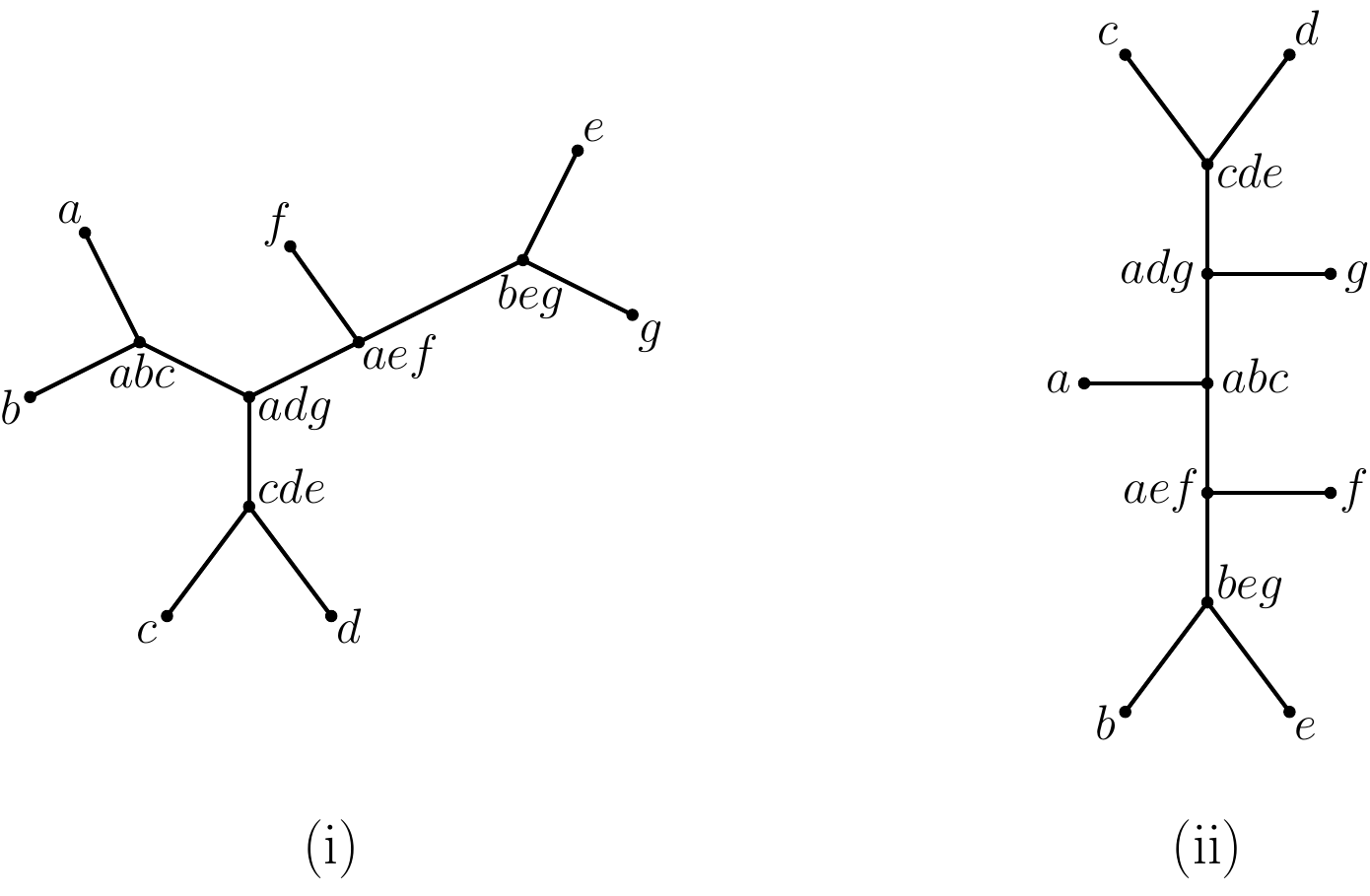}
\caption{(i) Associating each member of the thin collection of sets  $\{\{a,b,c\},   \{c, d, e\},  \{a,e, f\}, \{b, e, g\}, \{a, d, g\}\}$  with its median vertex in the tree shown provides a one-to-one mapping. (ii) A caterpillar tree that also  provides a median representation of this thin collection of sets.}
\label{fig3}
\end{figure}

We now strengthen this result from \cite {dre09} by showing that  the tree $T$ can always be chosen to be an unrooted caterpillar tree.  For example, for the thin collection of sets considered in Fig.~\ref{fig3}, we may select the caterpillar tree shown in Fig.~\ref{fig3}(ii). 

\begin{thm}
	\label{main3}
	Suppose $\tau$ is a non-empty subset of $\binom{X}{3}$, where $|X| \geq 4$.
	If $\tau$ is thin, then there exists an 
unrooted \underline{caterpillar}
tree $T=(V,E)$ with leaf set $X$ 
%and all vertices in $\Vr(T)$ of degree 3, 
for which the function 
${\rm med}_T:\tau \to \Vr(T)$ is one-to-one.
\end{thm}
\begin{proof}
  We adapt the proof of (3) $\Rightarrow$ (2) of \cite[Theorem 1.1]{dre09},  and use induction on the size of $X$. If $|X|=4$ the theorem clearly holds
  in view of Lemma~\ref{lemthin}(i). Let us suppose that  it holds 
whenever $4\leq |X|\leq n-1$, for some $n\geq 5$. Let $X$ be such that $|X|=n$.
By Lemma~\ref{lemthin}(ii), we may assume that 
one of the following two cases hold:
\begin{itemize}
\item[(A)] There is an element  $x$ of $X$ with $n_{\tau}(x)=1$.  
\item[(B)] There is an element  $x$ of $X$ with $n_{\tau}(x)=2$.
\end{itemize}
In case (A) there is some triple
$\{a,b,x\} \in \tau$ such that for $\tau' = \tau - \{\{a,b,x\}\}$ we have
that $\tau'$ is thin. Put $X'=L(\tau')$.  By induction, there is an 
unrooted caterpillar tree $T'$ with leaf-set $X'$ 
and the function ${\rm med}_{T'}: \tau'\to \Vr(T')$ 
is one-to-one. Now we can create a tree $T$ by inserting 
an edge $\{x,u\}$ where $u$ is a new vertex subdividing
an interior edge of $T'$ on the path 
between $a$ and $b$. The resulting tree $T$ is clearly a
unrooted caterpillar tree on $X$ 
and  ${\rm med}_T:\tau \to \Vr(T)$ is one-to-one. This
establishes the induction step in this case.\\

In Case (B) there is an element $x$ in $X$ with  $n_{\tau}(x)= 2$.
Then there exist two distinct triples $t, t'\in \tau$ each of
which contains $x$. We 
consider the following two possible cases:  (i) $|t\cap t'|=2$ and (ii) $|t\cap t'|=1$. \\
{\bf Case (i):}   $|t\cap t'|=2$. In this case, there
exist $a,b,b'\in X$ with $b\not= b'$ such that
$t=\{a,b,x\}$ and $t'=\{a,b',x\}$.
Since $\tau $ is thin, it follows that 
$$
\tau' = \tau - \{\{a,b,x\}, \{a,b',x\}\} \cup \{\{a,b,b'\}\}
$$ 
is also thin. Put $X'=L(\tau')$. Then, by induction, 
there is a unrooted caterpillar tree $T'$ with leaf-set $X'$ 
and ${\rm med}_{T'}:\tau'\to  \Vr(T')$ is one-to-one.

Consider the leaf $b'$ of $T'$. Let $b''\in \Vr(T')$ 
denote the vertex adjacent to
$b'$.  As $T'$ is an unrooted caterpillar tree, it suffices to
consider the following two subcases:\\

{\bf Subcase (a):} The leaves $a$ and $b$ are on the same side of 
$T'$ relative
to $b'$ (i.e. they are in the same connected component 
of $T'-b''$ as $b'$). 
Without loss of generality, assume that the distance from $a$ to $b'$ in 
$T'$ is less than or equal to distance from $b$ to $b$ in $T'$. 
Note that in this case ${\rm med}_{T'}(a,b,b')$ is the 
vertex in $T'$ that is adjacent to $a$. Now create a tree $T$
with leaf set  $X$ by inserting a new vertex $u$ and a new edge $\{u,x\}$ into $T'$ 
such that $\{u,b''\}$ is an edge on the 
%Previously: new edge $\{u,x\}$ into $T'$ where $u$ is 
%a new vertex subdividing the first
%interior edge in $T'$ that is closes to $b'$
% edge in $T'$ that is closest to $b'$ on the 
path connecting  $b'$ and $a$. The tree 
$T$ is again an unrooted  caterpillar tree on $X$. 
Furthermore, ${\rm med}_T:\tau\to \Vr(T)$
is one-to-one since (i) $\rm {\rm med}_{T'}$ is one-to-one,
and (ii) ${\rm med}_T(x,a,b')=u$  
and the median of $\{x,a,b\}$ in $T$
corresponds to the median vertex of $\{a,b,b'\}$ in $T'$ 
and therefore is a vertex of $T$ that is different from
any other median vertex of an element in $\tau$.

{\bf Subcase (b):} The leaves $a$ and $b$ are on different sides of
$T'$ relative to $b'$. Note that in this case,  ${\rm med}_{T'}(a,b,b')=b''$. 
Now  create a tree $T$
with leaf set $X$ by inserting a new vertex $u$ and a new
edge $\{x,u\}$ into $T'$ such that $\{u,b''\}$ is an edge 
%Previously: by subdividing by a new vertex $u$ the first  
%interior edge in $T'$ that is 
%closest to $b'$ 
on the path connecting  $b'$ and $b$.
$T$ is then clearly an unrooted  caterpillar tree on
$X$.
Since ${\rm med}_T(x,b',b)=u$ and the median of $\{x,a,b'\}$ in $T$
corresponds to the median vertex of $\{a,b,b'\}$ in $T'$, the same
arguments as in the previous case imply that
${\rm med}_T:\tau \to \Vr(T)$ is one-to-one.

{\bf Case (ii):}   $|t\cap t'|=1$. In this case, there exist
pairwise distinct element $a, a', b, b'$ in $X$ such that 
$t=\{a,b,x\}$ and $t'=\{a',b',x\}$. We may assume 
that $\tau$ does not contain both 
$\{a,a',b\}$ and $\{a,a',b'\}$ as, otherwise, the claim follows from
Case~(B)(i)(a).
By symmetry, we can assume without loss of generality that $\{a,b',b\}$ 
is not in $\tau$.
Let $\tau' = \tau - \{\{a,b,x\}, \{a',b',x\}\} \cup \{\{a,a',b\}\}$. Then
since $\tau$ is thin it follows that $\tau'$ is thin.
Put  $X'=L(\tau')$. Then, by induction, 
there is an unrooted  caterpillar tree $T'$ with leaf-set $X'$  and
${\rm med}_{T'}:\tau' \to \Vr(T')$ is one-to-one.
%which satisfies the conditions in the theorem for $\tau'$. 

Consider the leaf $a'$. 
%Let $a''\in\Vr(T')$ denote the vertex adjacent with $a'$. 
As $T'$ is a caterpillar tree on $X'$, 
we can again consider two subcases ((a) and (b)), the first of which involves two further subcases:\\

{\bf Case (a):} The leaves $a$ and $b$ are on the same side of $T'$ relative
to $b'$. Without loss of generality, assume that the distance
 from $a$ to $b'$ in 
$T'$ is less than or equal to distance from $b$ to $b'$ in $T'$. 
%Note that in this case, ${\rm med}_{T'}(a,a',b)=a''$. 
We now have two subcases to 
consider for this subcase:

%%%up to here

{\bf Case (a1):} The leaf $a'$ is on the same side of the caterpillar $T'$
as $a$ and $b$ relative to $b'$. If $a$ and $b$ are on the same side of
 $T'$ relative $a'$ then the same arguments as in the
Case~(B)(i)(a) apply with
$a'$ playing the role of $b'$.  If $a$ and $b$ are on 
different sides of $T'$ relative $a'$ then the same arguments 
apply as in  Case~(B)(i)(b) with $a'$ playing the role of $b'$.

% Details:
%Now create a tree $T$
%on $X$ by inserting a new edge $\{u,x\}$ into $T'$  
%by subdividing by a new vertex $u$ the first interior edge in $T'$ that is 
%closest to $b'$ on the path connecting  $b'$ and $a$. Thrn $T$ 
%is clearly an unrooted  caterpillar tree with leaf
%set $X$. 
%Since ${\rm med}_T(x,b',a')=u$
% and the median of $\{x,a,b\}$ in $T$
%corresponds to the median of $\{a,b',b\}$ in $T'$, 
%the same arguments as in the previous case imply that
%${\rm med}_T:\tau \to \Vr(T)$ is one-to-one.

{\bf Case (a2):} The leaf $a'$ is on a different side of the caterpillar $T'$
from $a$ and $b$ relative to $b'$. Now create a tree $T$
on $X$ by inserting a new vertex $u$ and a new edge $\{x,u\}$ 
into $T'$  such that with $b''\in \Vr(T)$ the vertex adjacent
with $b'$ we have that $\{b'',u\}$ is an edge 
%Previously by subdividing by a new vertex $u$ the first interior 
%edge in $T'$ that is closest to $b'$ 
on the path connecting  $a'$ and $b'$. Then
$T$ is clearly a unrooted caterpillar tree with leaf
set $X$.
Since ${\rm med}_T(x,a',b')$ is $u$ and the median of $\{x,a,b\}$ in $T$
corresponds to the median of $\{a,a',b\}$ in $T'$, 
the same arguments as in Case~(B)(i)(a) imply that
${\rm med}_T:\tau \to \Vr(T)$ is one-to-one.

{\bf Case (b):} The leaves $a$ and $b$ are on different sides of 
$T'$ relative
to $b'$.  If $a'$ lies on the same side of $T'$ as $a$ relative $b'$
and $a'$ and $b'$ lie on different sides of $T'$ relative $a$
then the same arguments as in the
Case~(B)(i)(a) apply with $a'$ playing the role of $b'$. In
all other cases the same arguments as in the  Case~(B)(i)(b)
apply with $a'$ playing the role of $b'$

%Details
%Note that in this case, ${\rm med}_{T'}(a,b',b)$ is the 
%vertex in $T'$ that is adjacent to $b'$.
%
%Without loss of generality, we may assume that $a'$ is on the same side of 
%$T'$ as $a$ relative to $b'$. Now create a tree $T$
%on $X$ by inserting a new edge $\{x,u\}$ into $T'$  
%by subdividing by a new vertex $u$ the first interior edge in $T'$ that is 
%closest to $b'$ on the path connecting  $b'$ and $b$. Then
%$T$ is then clearly an unrooted  caterpillar tree with leaf set $X$. 
%Since  ${\rm med}_T(x,a,b)=u$ 
% and the median of $\{x,a',b'\}$ in $T$
%corresponds to the median vertex of $\{a,a',b\}$ in $T'$,
%the same arguments as in the previous cases imply that 
% ${\rm med}_T:\tau \to \Vr(T)$ is one-to-one. 
\hfill$\Box$ \end{proof}

\subsection{The case $r=2$}\label{r2}

The concept of phylogenetic flexibility does not directly carry over to the case where $r=2$, since in this case, there is just a single rooted phylogenetic tree.  Instead, we use a stronger notion of tree structure (namely, total order) to obtain an analogue of Theorem~\ref{main}. 

We say that a non-empty subset $\tau$ of $\binom{X}{2}$ is {\em total-order flexible} if every choice of a  total order on the set  $s$, for each $s\in \tau$, is compatible with a total order on $X$. 
More formally, for every $s=\{x,y\} \in \tau$, if we declare  that either $x\prec y$ or $y\prec x$, then for any such selection of choices (one for each $s \in \tau$), there is a total order on $X$
that agrees with these inequalities. For example, $\tau= \{\{a,b\}, \{b, c\}\}$ is total-order flexible but $\{\{a,b\}, \{b, c\}, \{a,c\}\}$ is not, since the orderings $a\prec b, b \prec c, c \prec a$ are not compatible with any total order on $a,b,c$.
The following result is 
the analogue of Theorem~\ref{main} for the case where $r=2$

\begin{thm}
	Suppose that $\tau$ is a non-empty subset of $\binom{X}{2}$. Then   $\tau$ is thin if and only if $\tau$ is total-order flexible.
\end{thm}

\begin{proof}
	We first show that if $\tau$ is not thin, then $\tau$ is not 
total-order flexible. Suppose that $\tau$ is not thin. Then there exists 
a nonempty subset $\tau'$ of $\tau$ for which $|L(\tau')| \leq |\tau'|$. 
Let $G_{\tau'}$ be the graph $(L(\tau'), \tau')$ that has vertex set 
$L(\tau')$ and edge set $\tau'$.  Since $G_{\tau'}$ has at least as many 
edges as vertices, this graph has a connected component that contains 
a cycle. If the edges of this cycle are $\{x_1, x_2\}, \cdots, 
\{x_i, x_{i+1}\}, \cdots  ,\{x_r, x_1\}$, then the total orders 
$x_1\prec x_2, \cdots, x_i \prec x_{i+1}, \cdots , x_r \prec x_1$ on these 
pairs are not compatible with any total order on $X$ (since transitivity 
would imply that $x_1\prec x_1$).

	We now show that the thin property implies total-order flexibility 
by using induction on $k=|\tau|$.  The result clearly holds for $k=1$ so 
suppose that  the result holds for subsets of $\binom{X}{2}$ of size 
$k\geq 1$ and that $\tau \subseteq \binom{X}{2}$ is a thin set of size $k+1$.   By 
Lemma~\ref{lemthin}(ii), there is an element $x$ in $X$  that  is present 
in precisely one set, say $\{x,y\}$, in $\tau$. Let $\tau'$ be the set 
obtained from $\tau$ by deleting $\{x,y\}$.  Then $\tau'$ is thin and, since
	$|\tau'|=k$, the induction hypothesis implies that $\tau'$ is 
total-order flexible.  Then any choice of a total order on the set $s$ 
for each $s \in \tau'$ is compatible with a total order $\prec$ on $X-\{x\}$ 
(recall that $x\not\in L(\tau')$ by the choice of $x$).
	If we now introduce a total order on $\{x,y\}$, then we can extend 
the total order $\prec$ to $X$ by placing $x$ after $y$ if $\{x,y\}$ is 
ordered as $x,y$, and placing $x$ after $y$ otherwise. 
\hfill$\Box$ \end{proof}

We now present some characterizations for when a non-empty set $\tau \subseteq \binom{X}{2}$ is thin.
We begin with an analogue of \cite[Theorem 1.1]{dre09}, which was stated in the last section.

Given a rooted tree $T$ with leaf set $X$ and a set $s \in \binom{X}{2}$,  let 
${\rm lca}_T(s)$ refer to the vertex that is the unique  vertex of $T$ that is 
the least common ancestor of the elements in the set $s$.

\begin{thm}
	Suppose that $\tau$  is a subset of $\binom{X}{2}$ with $L(\tau) = X$. The following are equivalent:
	\begin{itemize}
		\item[(i)] $\tau$ is thin.
		\item[(ii)] There exists a  rooted binary phylogenetic  $X$--tree  $T=(V,E)$ for which the function $s \mapsto {\rm lca}_T(s)$ from the elements of $\tau$ to the set of interior vertices of $T$ is one-to-one.
		\item[(iii)] %There exists a rooted caterpillar tree $T=(V,E)$ with leaf set $X$ for which the function $s \mapsto {\rm lca}_T(s)$ from the elements of $\tau$ to the set of interior vertices of $T$ is one-to-one.
		As for (ii) but with $T$ a rooted caterpillar tree.
	\end{itemize}
\end{thm}

\begin{proof}

	(iii) $\Rightarrow$ (ii) is trivial.

	(i) $\Rightarrow$ (iii) Suppose $\tau$ is thin. We use induction on the size of $X$. 
	If $|X|=3$, then clearly (iii)  holds. Therefore,  suppose it holds 
	whenever $3\leq |X|\leq n-1$, for  some $n\geq 4$. Let $|X|=n$.

	By  Lemma~\ref{lemthin}(ii), there is some $x$ with $n_{\tau}(x)=1$.   
	Let $X' = X - \{x\}$. It is then straightforward to see that 
	there is some pair  $\{a,x\} \in \tau$ with 
$\tau' = \tau - \{\{a,x\}\}$. Clearly $\tau'$ is thin as $\tau$ is thin.
	and either $L (\tau') = X-\{x\}$ or $L (\tau') = X-\{x,a\}$.

Assume first that $L (\tau') =X'$, where $X'= X-\{x\}$.
By induction, there is a 
	rooted caterpillar tree $T'$ with leaf-set $X'$ and root $\rho'$ for which the function ${\rm lca}_{T'}: \tau'\to \Vr(T')$ 
	is one-to-one. Now,  we can create a new rooted tree $T$ with root $\rho$ by 
	adding two new edges $\{\rho,\rho'\}$ and $\{\rho,x\}$ to $T$ 
	where $\rho$ is a new vertex that is not in $T'$.
	$T$ is then clearly a rooted caterpillar tree on $X$, 
	every vertex in $\Vr(T)$ has out-degree 2
	and  ${\rm lca}_T:\tau \to \Vr(T)$ is one-to-one. This
	establishes the induction step, and so (iii) holds.

		Assume next that $L (\tau') = X'$, where $X'=X-\{x,a\}$.
By induction, there exists a rooted caterpillar tree
$T'$ on $X'$. Let $\rho'$ denote the root of $T'$. Let $T$ be rooted
caterpillar tree obtained from $T'$ via the following 2 step
process. First, add a new root $\rho$ and a new edge
$e=\{\rho,\rho'\}$ to $T'$. In the resulting tree subdivide
$e$ by a vertex $c$ and add the edges $\{c,a\}$ and
$\{\rho,x\}$. Clearly, ${\rm lca}_T:\tau \to \Vr(T)$ is one-to-one.
This establishes again the induction step, and so (iii)
holds too in this case.

	(ii) $\Rightarrow$ (i)  Suppose $x$ is an element which is not in 
	$L(\tau) = X$. Given  a non-empty subset $\omega$ of $\tau$,
	 let $\omega^*= \{ t \cup \{x\} \, : \, t \in \omega \}$.

	Suppose  a rooted phylogenetic  tree $T$ on $X$
	 satisfies the conditions in Part (ii) of 	the theorem. Add a new
	leaf $x$ that is not in $L(\tau)$ to $T$ by adding the edge
	$\{\rho_T,x\}$ and regard the resulting tree as an unrooted phylogenetic tree $T'$
	on $X \cup \{x\}$. In $T'$, the map 
	${\rm med}_{T'}$ from $\tau^*$ to the internal vertices of $T'$ is then one-to-one. 
	Hence  by \cite[Theorem 1.1]{dre09},  $\tau^*$ is thin and thus for any non-empty subset
	$\omega$ of $\tau$, we have:
	$$
	|L( \omega) | + 1 = |L( \omega^*)| \ge | \omega^*| + 2 = |\omega| + 2.
	$$
	It immediately follows that $\tau$ is thin.

\hfill$\Box$ \end{proof}

Interestingly, we can give an  alternative characterisation of thin subsets $\tau$ of $\binom{X}{2}$
in terms of bipartite graphs.

We first recall some results from matching theory.
For a graph $G$ and $v$ a vertex in $G$, we let ${\rm deg}_G(v)$ denote the degree of $v$ in $G$.
Given a bipartite graph $G=(A\cup B,E)$ and a non-empty set $Y \subseteq A$, 
we let $N_G(Y)$ denote the set of vertices in $B$ that are adjacent to some vertex in $Y$,
and we define the 
{\em surplus} $\sigma_G(Y)$  of $Y$ to be:
$$
\sigma_G(Y) =  \sigma(Y) = |N_G(Y)| - |Y|.
$$
We also define the {\em surplus} $\sigma(G)$ of $G$, to be the
minimum surplus over all non-empty sets of $A$. We 
say that a bipartite graph $G=(A\cup B,E)$ has {\em positive surplus (as viewed from $A$)} if $\sigma(G)>0$.
The following result is  from Lov$\grave{\mbox{a}}$sz and Plummer \cite[Theorem 1.3.8]{LP}. 

\begin{thm}\label{forest}
	A bipartite graph $G=(A\cup B,E)$ has positive surplus
	(as viewed from $A$) if and only if $G$ contains a forest $F$ such that ${\rm deg}_F (u) = 2$ for
	all $u \in  A$.  
\end{thm}

We now apply this result to the setting of thin sets.
Let $\tau$ be a non-empty collection of non-empty subsets of $X$. We associate a 
bipartite graph $G(\tau)$ to $\tau$ that has the vertex set $\tau \cup L(\tau)$
and the edge set given by containment (i.e. $\{t,x\}$ is an edge in $G(\tau)$
if and only if $x \in t$ with $x\in L(\tau)$ and $t\in\tau$). Thus we are representing our set $\tau$ by 
a bipartite graph $G=(A\cup B,E)$ with $A=\tau$, $B=X$ and $E$
given  by containment. 
%Note that if $\sigma(G(\tau))=k$,
%then $|L(\tau') | \ge |\tau'| +k$ for all non-empty subsets $\tau'$ of $\tau$, and
%there exists some $\tau'$ with $|L(\tau') | = |\tau'| +k$. 

Since $\tau$ is thin if and only if $G(\tau)$ has positive surplus, by Theorem~\ref{forest}, the following corollary is straightforward.

\begin{corollary}
	Suppose that $\tau$  is a subset of $\binom{X}{2}$ with $L(\tau) = X$. Then 
	$\tau$ is thin if and only if $G(\tau)$ is a forest. 
\end{corollary}

For $r=3$, it might also be  interesting to characterise those graphs $G(\tau)$ for which $\tau$ is thin.

\section{The general case (Slim set systems)}
\label{slimcase}

Suppose we have a non-empty collection $\tau$ of subsets of $X$, each of size at least 3. Consider the modified notion of excess, denoted $\x'$ and defined as follows:
Define $$\x'(\tau) =  |L(\tau)|-2 -  \sum_{s \in \tau} (|s|-2).$$

Notice that when  $\tau \subseteq \binom{X}{3}$ this notion of excess agrees with the earlier one. 

Given a non-empty collection $\tau$ of subsets of $X$, each of size at least 3, we say that $\tau$ is 
is {\em slim} if for every non-empty subset $\tau'$
of $\tau$, we have $\x'(\tau')\geq 0$. 
The next result relates slim to thin; the two notions coincide when $\tau \subseteq \binom{X}{3}$, however, slim is a more restrictive notion than thin when $\tau \subseteq \binom{X}{r}$,
for $r>3$. Note, however, that (unlike the thin property) the slim property does not require the  sets in $\tau$ to all have the same size.

\begin{lemma}
Suppose that $\tau \subseteq \binom{X}{r}$, $r \geq 3$. If $\tau$ is slim, then $\tau$ is thin.  Moreover, for $r=3$,
$\tau$ is thin if and only if $\tau$ is slim.
\end{lemma}

\begin{proof}
If $|s|=r$ for each $s \in \tau$, then $\sum_{s \in \tau'} (|s|-2) = |\tau'|(r-2)$; therefore,  $\tau$ is thin if and only if
$|L(\tau')| \geq |\tau'|(r-2) +2$ for every non-empty subset $\tau'$ of $\tau$. We now impose the assumption that $r\geq 3$. First, if $r>3$,  then the required inequality:  $|\tau'|(r-2) +2 \geq |\tau'|+(r-1)$ is equivalent to the condition that $|\tau'| \geq 1$, which holds
by the assumption that $\tau'$ is non-empty.  Thus if $\tau$ is slim, it is also thin.  Moreover, when $r=3$, the inequality
$|\tau'|(r-2) +2 \geq |\tau'|+(r-1)$ becomes an equality; in this case, $\tau$ is slim if and only if it is thin. 

\hfill$\Box$ \end{proof}

We can extend the notion of phylogenetic flexibility introduced in Section~\ref{intro}
to arbitrary collections of subsets of $X$ as follows.
We first need to extend the earlier definitions of `display'
and `compatibility' from
sets of rooted triples to arbitrary collections of rooted trees, as follows.

Given a rooted phylogenetic $X$-tree $T$, and a binary phylogenetic tree
$T'$ with leaf set $Y \subseteq X$,
$T$ is said to {\em display} $T'$ if $T$ contains a subdivision of $T'$ as a (directed)
subtree (this is equivalent to the condition that
each rooted triple  displayed by $T'$ is also displayed by $T$ \cite{bry}).
A set $R$ of rooted binary phylogenetic trees is said to be {\em compatible} if there is
rooted phylogenetic tree $T$ that displays each of the trees in $R$.

For a set $R$ of rooted binary phylogenetic tree, let $||R||$ denote the collection of
their leaf sets. Thus $||R||$ is a set of sets. Given a non-empty collection
$\tau$ of  subsets of $X$, each of size at least 3, we say that $\tau$ is {\em
phylogenetically flexible} if every set $R$ of rooted binary phylogenetic trees for
which $||R|| = \tau$ holds is compatible.

This notion agrees with the earlier notion of phylogenetic flexibility in the case where each set in $\tau$ has size exactly 3.
Moreover, as before, we can assume without loss of generality that the tree $T$ (in the definition) is binary.

The following result is a strengthening of our earlier Theorem~\ref{main}; one direction follows from that theorem, the other direction is a consequence of a result from \cite{gru12} (which dealt with unrooted trees).

\begin{thm}
\label{propo}
Suppose that $\tau$ is a  collection of sets, each of size at least 3.  Then $\tau$ is phylogenetically flexible if and only if $\tau$ is slim.
\end{thm}

\begin{proof}
We first establish the `only if' direction.
Suppose that $\tau$ is phylogenetically flexible.  For each set $s \in \tau$, select two elements $x,y \in s$ and let:
$$A(s) = \{\{x,y,z\}: z \in s, z\neq x,y\}\},$$
and  for any non-empty subset $\tau'$ of $\tau$ let 
$$
\alpha(\tau')= \bigcup_{s \in \tau'} A(s).
$$
Thus $A(s)$ is a set of $|s|-2$ triples, and $\alpha(\tau)$ is also a set of triples.
%yyyy

\noindent {\bf Claim 1}:  $A(s) \cap A(s') = \emptyset$ for each $s,s'\in \tau,  s \neq s'$.

To see this, suppose that a triple, say $\{a,b,c\}$, lies in $A(s)$ and $A(s')$ for two distinct elements
$s$ and $s'$ of $\tau$.  Then we can select a rooted binary phylogenetic tree $T_s$ with leaf set $s$ that displays the rooted triple $ab|c$, and select a rooted binary phylogenetic tree $T_{s'}$  with leaf set $s'$ that displays the rooted triple $ac|b$.  But no rooted binary phylogenetic tree can display both $T_s$ and $T_{s'}$ (since such a tree would also simultaneously display two different rooted triples with leaf set $\{a,b,c\}$). This contradicts the assumption that $\tau$ is phylogenetically flexible, so such a shared triple $\{a,b,c\}$ in $A(s) \cap A(s')$ cannot exist.
This establishes Claim 1.

\noindent {\bf Claim 2}:  $\alpha(\tau)$ is phylogenetically flexible.

To see this, suppose that for each triple $t \in \alpha(\tau)$, we have an associated rooted triple $T_t$  with leaf set $t$.
We need to show that there is a rooted binary phylogenetic tree that displays $\{T_t: t \in \alpha(\tau)\}$.
Observe that $A(s)$ is thin for each $s\in \tau$, since if $A$ is a non-empty subset of $A(s)$ of size $k$ (say), then 
$|\bigcup A| = k+2$, and so $|\bigcup A| = |A|+2$. 
 Theorem~\ref{main} (the `if' direction) then ensures that for each $s$ in $\tau$, there is a rooted phylogenetic tree
$T_s$ with leaf set $s$ that displays $\{T_t: t \in \alpha(\tau) \cap A(s)\}$.  Moreover,  since $\tau$ is phylogenetically flexible, there is a rooted binary phylogenetic  tree $T$  that displays $T_s$ for each $s\in \tau$.  It follows that the tree $T$ displays $\{T_t: t \in \alpha(\tau)\}$, and so $\alpha(\tau)$ is phylogenetically flexible, as claimed.

Claim 2 implies that $\alpha(\tau)$ is thin, by Theorem~\ref{main} (the `only if' direction). We now show that this implies that
$\tau$ is slim.    Let $\tau'$ be a non-empty subset of $\tau$ and consider $\alpha(\tau')$.  
Since  $L(\tau')=L(\alpha(\tau'))$, we have:
\begin{equation}
\label{cuptau}
|L(\tau')| = |L( \alpha(\tau'))| \geq |\alpha(\tau')| +2,
\end{equation}
where the inequality holds  because $\alpha(\tau)$ (and thereby its subset
$\alpha(\tau')$) is thin.

Now:
 $$|\alpha(\tau')| = \sum_{s \in \tau'} |A(s)| = \sum_{s \in \tau'} (|s|-2),$$
 where the first equality holds by Claim 1.
Combining this last equation with  Inequality~(\ref{cuptau}) gives:
$$|L(\tau')|-2 \geq  \sum_{s \in \tau'}( |s|-2),$$ which shows that $\tau$ is slim as claimed.

This establishes the `only if' direction.   Notice in doing so that we have used {\em both} directions of Theorem~\ref{main} in different places in this proof.

We turn now to the `if'  direction.
Given  $\tau$, select  a new element, say $x$, that is not present in any of the sets in $\tau$, and add this to each of the sets in $\tau$ to produce a set $\tau_{+x}$.  Notice that if $\tau$ is slim, then $\tau_{+x}$ satisfies the property that for each non-empty set $\tau'$ of $\tau_{+x}$, we have:
$$|L(\tau')| -3 \geq \sum_{s\in \tau'}(|s|-3).$$
It follows from Theorem 1.1 of \cite{gru12} that for any assignment of unrooted binary phylogenetic trees with leaf sets that correspond to the sets in $ \tau_{+x}$, there is a binary phylogenetic tree $T_x$ that displays each of these unrooted trees. 
Suppose now that we have an assignment of rooted binary phylogenetic trees having leaf sets that correspond to the sets in $\tau$.
By attaching $x$ as a leaf adjacent to the root of each of these trees, we obtain an assignment of unrooted binary phylogenetic trees with leaf sets that correspond to the sets in $\tau_{+x}$.  
Hence, by the result just stated, there is an unrooted binary phylogenetic tree $T_x$ that displays each of these unrooted trees. If we now let $T$ be the rooted binary phylogenetic tree obtained  from $T_x$ by deleting the leaf $x$ and rooting the resulting tree on the vertex adjacent to $x$, then $T$ displays the original assignment of rooted binary phylogenetic trees.  Since this holds for all possible assignments of rooted phylogenetic trees to the sets in $\tau$, it follows that $\tau$ is phylogenetically flexible. 
\hfill$\Box$ \end{proof}

\section{Polynomial--time algorithms for thin and slim}

Given  finite set $S$, a function $f: 2^S \to \R$ is called {\em submodular} if for all $A,B \subseteq S$:
$$
f(A)+f(B) \ge f(A \cup B) + f(A \cap B).
$$
Submodular functions play an important role in optimization and matroid theory (see
e.g. \cite{LP,B,W}). In this section, we exploit these connections to show that 
there are polynomial-time algorithms to decide if sets are thin or slim.

Suppose that $\tau$ is a subset of $2^X$. 
For $\tau' \subseteq \tau$ we define 
$$
\sigma(\tau') = |L(\tau')| - |\tau'|,
$$
and
$$
\gamma(\tau') = |L(\tau')| - \sum_{s \in \tau'} (|s|-2)
$$
Note that $\sigma(\emptyset)=\gamma(\emptyset)=0$
(since the summation term is then empty).
Although the following result is straight-forward to show by using results
concerning submodular functions in the literature, for completeness
we give a direct proof.

\begin{thm}\label{submod}
	For $\tau$ a non-empty  subset of $2^X$, 
the functions $\sigma:2^{\tau} \to \R; \tau' \mapsto \sigma(\tau')$ 
	and $\gamma:2^{\tau}\to \R; \tau' \mapsto \gamma(\tau')$ 
are submodular. 
\end{thm}
\begin{proof}
Suppose that $\tau$ is a non-empty subset of $2^X$, and that 
$\tau_1, \tau_2 \subseteq \tau$ are non-empty.
Clearly, $|L(\tau_1 \cup \tau_2)| = |L(\tau_1)  \cup L(\tau_2)|$ and
$|L(\tau_1 \cap \tau_2)| \le  |L(\tau_1) \cap L(\tau_2)|$. Hence:
\begin{eqnarray*}
|L(\tau_1 \cup \tau_2)| + |L(\tau_1 \cap \tau_2)| & \le & |L(\tau_1)  \cup L(\tau_2)| + |L(\tau_1) \cap L(\tau_2)|\\
& = & (|L(\tau_1)| + |L(\tau_2)|  -  |L(\tau_1) \cap L(\tau_2)| )+ |L(\tau_1) \cap L(\tau_2)|\\
& = & |L(\tau_1)| + |L(\tau_2)|.
\end{eqnarray*}
The fact that $\sigma$ is submodular now follows, since
$|\tau_1| + |\tau_2|  = |\tau_1 \cup \tau_2| + |\tau_1 \cap \tau_2|$ and thus, by the above inequality, we have:
$$\sigma(\tau_1) + \sigma(\tau_2) =
|L(\tau_1)| + |L(\tau_2)| -(|\tau_1| + |\tau_2|) \ge $$
$$|L(\tau_1 \cup \tau_2)| + |L(\tau_1 \cap \tau_2)| - |\tau_1 \cup \tau_2| - |\tau_1 \cap \tau_2| 
= \sigma(\tau_1 \cup \tau_2) + \sigma(\tau_1 \cap \tau_2).
$$
Similarly, $\gamma$  is submodular, since 
$$\sum_{s \in \tau_1} (|s|-2) + \sum_{s \in \tau_2} (|s|-2) =  
\sum_{s \in \tau_1\cup \tau_2} (|s|-2) + \sum_{s \in \tau_1 \cap \tau_2} (|s|-2)$$ and therefore:
\begin{eqnarray*}
\gamma(\tau_1) + \gamma(\tau_2) &= &
|L(\tau_1)| + |L(\tau_2)| -(\sum_{s \in \tau_1} (|s|-2) + \sum_{s \in \tau_2} (|s|-2) )\\
& \ge &
|L(\tau_1 \cup \tau_2)| + |L(\tau_1 \cap \tau_2)| - \sum_{s \in \tau_1\cup \tau_2} (|s|-2) - \sum_{s \in \tau_1 \cap \tau_2} (|s|-2)\\
&=& \gamma(\tau_1 \cup \tau_2) + \gamma(\tau_1 \cap \tau_2).
\end{eqnarray*}
\hfill$\Box$ \end{proof}

For $\tau \subseteq 2^X$, we define: 
$$\sigma^*(\tau) = \min\{\sigma(\tau') \,:\, \tau' \subseteq \tau, \tau' \neq \emptyset\},
\mbox{ and } \gamma^*(\tau) = \min\{\gamma(\tau') \,:\, \tau' \subseteq \tau, \tau' \neq \emptyset\}.$$

\begin{lemma}\label{obs}
\mbox{}
\begin{itemize}
	\item[(i)] 
	Suppose that $\tau \subseteq {X \choose r}$ where $r\geq 3$.  
Then $\tau$ is thin if and only $\sigma^*(\tau) \ge 2$.
	\item [(ii)] Suppose $\tau \subseteq 2^X$ such that each element in $\tau$ has size at least three. Then $\tau$ is slim
		  if and only $\gamma^*(\tau) \ge 2$.
	\end{itemize}	
\end{lemma}
\begin{proof}
(i) $\tau$ is thin if and only if $\sigma(\tau') \ge 2$ for all non-empty $\tau' \subseteq \tau$ 
if and only if $\sigma^*(\tau) \ge 2$.\\

(ii)  $\tau$ is slim if and only if $\gamma(\tau') \ge 2$ for all non-empty $\tau' \subseteq \tau$ 
if and only if $\gamma^*(\tau) \ge 2$.
\hfill$\Box$ \end{proof}

The following result  \cite[Theorem 4.4]{lov} is originally due to Gr\"otschel,  Lov{\'a}sz
and Schrijver \cite{GLS} (see also \cite[pp. 417--418]{LP}).

\begin{thm}\label{gls}
	Let $f$ be a submodular function defined on the subsets of 
some finite set $S$.
	A set minimizing $f$ over all non-empty subsets of $S$ can then be found in polynomial time.
\end{thm}

In light of this theorem and Theorem~\ref{submod}, it follows that we can determine
$\sigma^*(\tau) $ and $\gamma^*(\tau)$  for a given set $\tau \subseteq 2^X$ 
in polynomial time. Therefore, by 
Lemma~\ref{obs}, we can determine whether or not a given set $\tau$ for which each
element has size at least three is thin or slim in polynomial time.

Note that although this shows that polynomial time algorithms exist for determining whether or not a set is thin or slim, these are likely to be impracticable
\cite[pp. 417--418]{LP}. However,  
for the case of determining whether or not a set $\tau$ is thin a more explicit
algorithm can be given. More specifically, in 
\cite[Theorem 2]{F13} a polynomial-time algorithm is presented for 
computing the surplus $\sigma(G)$ of a bipartite graph 
$G$. Since for a set $\tau \subseteq {X \choose r}$, $r\geq 3$, 
the surplus of the 
bipartite graph $G(\tau)$ as defined in Section~\ref{r2}
is equal to $\sigma^*(\tau)$, we can therefore 
apply this algorithm to determine if $\tau$ is thin. It would
be interesting to find an explicit algorithm for determining if a set is slim.

Theorem~\ref{submod} has another consequence that relates to phylogenetics. 
Recall that a {\em patchwork} is a non-empty collection ${\mathcal P}$ of sets that satisfies the property:  if $A, B \in {\mathcal P}$ and $A \cap B \neq \emptyset$  then $A \cap B, A \cup B \in {\mathcal P}$.  A combinatorial theory of patchworks, relevant to phylogenetics, was developed in  \cite{boc}.  Patchworks were also referred to as `intersecting families' in earlier work by Lov{\'a}sz in \cite[p. 240]{lov}.  	The following is a generalisation of \cite[Lemma 1.2]{dre09}, and  the proof follows a similar argument to that result.

\begin{corollary}

If $\tau$ is slim, then the collection $\mathcal P$ of non-empty subsets 
$\tau'$ of $\tau$ such that $|s|\geq 3$ for all $s\in \tau'$ 
and $\x'(\tau')=0$ forms a patchwork.
	\end{corollary}

	\begin{proof}
Suppose $\tau_1, \tau_2 \in \mathcal P$ satisfy $\tau_1 \cap \tau_2 \neq \emptyset$. For $i=1,2$, notice that $\x'(\tau_i) = \gamma(\tau_i) -2$ and so, 
by the submodularity property of $\gamma$ from Theorem~\ref{submod}, we have:
\begin{equation}
\label{au1}
\x'(\tau_1)+ \x'(\tau_2) \geq \x'(\tau_1 \cup \tau_2) + \x'(\tau_1 \cap \tau_2),
\end{equation}
noting that $\x'(\tau_1 \cap \tau_2)$ is well defined by the condition that $\tau_1 \cap \tau_2 \neq \emptyset$.
Since $\x'(\tau_1)=\x'(\tau_2)=0$, Inequality~(\ref{au1}) gives:
$$
0 \geq \x'(\tau_1 \cup \tau_2) + \x'(\tau_1 \cap \tau_2).
$$
It follows that the terms $\x'(\tau_1 \cup \tau_2)$ and 
$\x'(\tau_1 \cap \tau_2)$
on the right of this inequality  must both be zero since 
$\tau_1 \cup \tau_2$ and $\tau_1 \cap \tau_2$ are non-empty subsets of the slim set $\tau$ and so each has non-negative excess.
Thus, $\tau_1 \cup \tau_2, \tau_1 \cap \tau_2 \in \mathcal P$, as required.

	\hfill$\Box$ \end{proof}

\section{Discussion}

When  an evolutionary biologist compares a number of trees on different, but overlapping, leaf sets it is typically  very rare that these trees are found to be compatible, due mainly to errors in the estimation of phylogenetic trees.   Thus, in cases where the trees are compatible this fact alone may provide the biologist with some heightened confidence in the accuracy of the input trees.  However, such confidence should clearly depend, in part, on the pattern of taxon coverage.  In the extreme case where the subsets of species on which the input trees were built form a phylogenetically flexible collection, it is clear that compatibility provides absolutely no hint of accuracy of the input trees, since {\em any} trees that had been considered for those subsets would be compatible. 
For applications, it might therefore  be useful to quantify how close to `phylogenetically flexible' a given pattern of taxon coverage is. 

Our results also suggest a second possible future research direction. Since submodular functions are connected to matroid theory, are there 
relevant connections between thin/slim sets and matroids? Other matroid structures in phylogenetics have been recently been described, in different contexts,  by \cite{dre14} and  \cite{hel}.

\section{Acknowledgments}
We thank the organisers of the Algebraic and Combinatorial Phylogenetics Workshop (Barcelona, June 26--30 2017)
where some of the ideas in this paper were conceived, and the London Mathematical
Society for supporting the visit of KTH and VM to visit MS in New Zealand.  We also thank the two anonymous reviewers of this paper
for numerous helpful suggestions.

%\section{Appendix}

\end{document}